\documentclass[11pt,twoside,a4paper]{amsart}
\usepackage{amssymb}
\date{\today}

\def\sko{{\vartheta}}

\def\z{\zeta}
\def\Pn{{\mathbb P}^N}

\def\deg{\text{deg}\,}

\def\w{\wedge}

\def\dbar{\bar\partial}

\def\C{{\mathbb C}}
\def\w{{\wedge}}
\def\Kers{{\mathcal Ker\, }}
\def\P{{\mathbb P}}
\def\W{{\mathcal W}}
\def\Pk{{\mathbb P}}

\def\A{{\mathcal A}}

\def\S{{\mathcal S}}

\def\Cu{{\mathcal C}}

\def\Pr{{\mathcal P}}
\def\K{{\mathcal K}}

\def\Hom{{\rm Hom\, }}
\def\codim{{\rm codim\,}}

\def\Im{{\rm Im\, }}

\def\K{{\mathcal K}}
\def\Ker{{\rm Ker\,  }}

\def\E{{\mathcal E}}

\def\Ok{{\mathcal O}}

\def\L{{\mathcal L}}

\def\Re{{\rm Re\,  }}
\def\1{{\bf 1}}
\def\reg{{\rm reg\,  }}

\def\L{{\mathcal L}}
\def\U{{\mathcal U}}

\def\J{{\mathcal J}}
\def\pmm{pseudomeromorphic}
\def\nbh{neighborhood }
\def\PM{{\mathcal{PM}}}

\def\be{\begin{equation}}
\def\ee{\end{equation}}
\def\Cu{{\mathcal C}}

\newtheorem{thm}{Theorem}[section]
\newtheorem{lma}[thm]{Lemma}
\newtheorem{cor}[thm]{Corollary}
\newtheorem{prop}[thm]{Proposition}

\theoremstyle{definition}

\theoremstyle{remark}

\newtheorem{preremark}{Remark}
\newtheorem{preex}{Example}

\newenvironment{remark}{\begin{preremark}}{\qed\end{preremark}}
\newenvironment{ex}{\begin{preex}}{\qed\end{preex}}

\numberwithin{equation}{section}

\title[]{Global Koppelman formulas on (singular) projective varieties}

\begin{document}

\date{\today}

\author{Mats Andersson}

\address{Department of Mathematics\\Chalmers University of Technology and the University of 
G\"oteborg\\S-412 96 G\"OTEBORG\\SWEDEN}

\email{matsa@math.chalmers.se}


\thanks{The  author was
  partially supported by the Swedish 
  Research Council}

\begin{abstract}
Let $i\colon X\to \Pk^N$ be a  projective manifold of dimension $n$ embedded in projective space $\Pk^N$,
and let $L$ be the pull-back to $X$ of the line bundle $\Ok_{\Pk^N}(1)$.
We construct global explicit Koppelman formulas on $X$ for smooth $(0,*)$-forms with values  in $L^s$ for any $s$.  
The same construction works for singular, even
non-reduced, $X$ of pure dimension, if the sheaves of smooth forms are replaced by
suitable sheaves $\A_X^*$ of $(0,*)$-currents with mild singularities at $X_{sing}$.
In particular, if $s\ge \reg X -1$, 
where $\reg X$ is the Castelnuovo-Mumford regularity, we get an explicit
 representation of the well-known vanishing of $H^{0,q}(X, L^{s-q})$, $q\ge 1$.
Also some other applications are indicated.
  \end{abstract}

\maketitle

\section{Introduction}

During the last decade global Koppelman formulas for $\dbar$ on various special projective varieties have been constructed, see, e.g., \cite{EG1, GSS,Helge, HePo2,HePo3,HePo4}.
 The aim of this paper is to present a quite general explicit construction of 
intrinsic\footnote{Here "intrinsic"  means that the operators in the formula
only depend on intrinsic forms on $X$.}  
Koppelman formulas on any projective, possibly non-reduced, subvariety 
$i\colon X\to \Pk^N$ of pure dimension $n$.

\smallskip
Let us first assume that $X$ is smooth; even in this case such global
formulas are previously known only in case $X$ is (locally) a complete intersection
in $\Pk^N$.   Let $L\to X$ be the restriction of the ample line bundle $\Ok_{\P^N}(1)$ to $X$, and let
 $\E_X^{0,q}(L^r)$ denote the sheaf of smooth $(0,q)$-forms on $X$ with values
in $L^r$.
We introduce integral operators 
$\K\colon\E^{0,q+1}(X,L^s)\to\E^{0,q}(X_{reg},L^s)$ and 
$\Pr\colon\E^{0,q}(X,L^s)\to\E^{0,q}(X,L^s)$
such that the Koppelman formula
\begin{equation}\label{koppelman1}
\phi(z)=\dbar\K\phi+\K(\dbar\phi)+\Pr\phi  
\end{equation}
holds on $X$.
In some situations, see below, we can choose the operators
so that  $\Pr\phi=0$; if $\dbar\phi=0$, then 
$\psi=\K\phi$ is a smooth solution to $\dbar\psi=\phi$ on $X$.
In certain cases we can choose $\Pr$ such that
$\Pr\colon\E^{0,0}(X,L^s)\to\Ok(\P^N,\Ok(s))$. Then $\Pr\phi$
is a holomorphic extension to $\Pk^N$ of a holomorphic section $\phi$ 
of $L^r$ on $X$.
We get no new existence results; the novelty is that we
have explicit formulas for the global solutions and for the holomorphic extensions.
The operators $\K$ are given by  kernels $k(\zeta,z)$ 
that are defined 
on  $X\times X$ and  integrable in $\zeta$ for any $z\in X$. Simply speaking
the operators locally behave like standard integral operators for $\dbar$ in
$\C^n$; in particular they extend to $L^p$-spaces etc and 
all classical local norm estimates hold.

\smallskip
Let us now turn our attention to the case when 
$i\colon X\to \Pk^N$ is a subvariety of pure dimension $n$.  
In case $X$ is reduced, there are well-known definitions of smooth forms and
currents on $X$, cf.,  Section \ref{residue}. 
In \cite{AL} are introduced reasonable definitions of sheaves $\E^{0,*}_X$
of smooth $(0,*)$-forms and suitable sheaves of currents
even in the non-reduced case, see Section \ref{nonred} below.
In \cite{AS1} and \cite{AL} are introduced, by means of local Koppelman formulas,
fine  sheaves $\A^k_X$ (in fact, modules over $\E^{0,*}_X$) of $(0,q)$-currents  on $X$, or any pure-dimensional analytic space, with the following properties:
There are sheaf inclusions $\E_X^{0,q}\subset \A_X^q$ with equality on 
$X_{reg}$, whereas $\A_X^q$  have ``mild'' singularities
at $X_{sing}$, and 
\begin{equation}\label{upp}
0\to\Ok\to\A_X^0\stackrel{\dbar}{\to}\A_X^1\stackrel{\dbar}{\to}\cdots
\end{equation}
is a (fine) resolution of the structure sheaf $\Ok_X$ of holomorphic functions on $X$.
By the abstract de~Rham theorem we therefore have canonical isomorphisms
\begin{equation}\label{isos}
H^k(X,L^s)\simeq \frac{\Ker(\A^k(X,L^s)
\stackrel{\dbar}{\to} \A^{k+1}(X,L^s))}
{\Im (\A^{k-1}(X,L^s))\stackrel{\dbar}{\to} \A^{k}(X, L^s))},\quad k\ge 1.
\end{equation}

In this paper we construct integral operators 
$$
\K\colon\A^{q+1}(X,L^s)\to\A^{q}(X,L^s), \quad 
\Pr\colon\A^{q}(X,L^s)\to\E ^{0,q}(X,L^s)
$$
such that again the Koppelman formula \eqref{koppelman1}
holds on $X$.
In the reduced case,  the operators $\K$ are given by  kernels $k(\zeta,z)$ that are defined 
on  $X_{reg}\times X$, locally integrable in $\zeta$ for $z\in X_{reg}$,
$\Pr$ are given by  kernels $p(\zeta,z)$ that are smooth
on $X_{reg}\times X$, and  the integrals 
$$
\K \phi(z)=\int_X k(\zeta,z)\w\phi(\zeta),  \  z\in X_{reg}, \quad 
\Pr \phi(z)=\int_X p(\zeta,z)\w\phi(\zeta), \   z\in  X,
$$
exist as principal values at $X_{sing}$. For the non-reduced case, 
see Section \ref{nonred}. 
As in the smooth case,  in good situations $\Pr\phi$ vanishes 
so we get explicit solutions to $\dbar\psi=\phi$,  and 
extensions of holomorphic sections from $X$ to $\Pk^N$.

\begin{remark}
It was proved already in \cite{HPo} that if $X$ is a local reduced  
complete intersection
and 
$\phi$ is a smooth $\dbar$-closed form, then there is locally a smooth solution to
$\dbar\psi=\phi$ on $X_{reg}$. The case with a general  $X$ was proved only in \cite{AS0}.
The analogous result for a local non-reduced space is a special case
of the main result in \cite{AL}. 
It is known that in  general there is no solution that is smooth across $X_{sing}$, see,
e.g., \cite[Example 1.1]{AS1}. 
\end{remark}

Let $J_X$ be the homogeneous ideal in the graded ring $S=\C[z_0,\ldots,z_N]$ 
associated with $X$.  Let  $S(-r)$ be the module $S$ but  with the grading shifted by $r$.
There is a free graded resolution 
\begin{equation}\label{kongo}
0\to M_{N_0}\stackrel{a_{N_0}}{\longrightarrow}\ldots
  \stackrel{a_3}{\longrightarrow} M_2     \stackrel{a_2}{\longrightarrow} M_1   
\stackrel{a_1}{\longrightarrow} M_0       
\end{equation}
of the homogeneous module $M_J:=S/J_X$; i.e., 
$$
M_0=S, \quad M_k=S(-d_k^1)\oplus\cdots\oplus S(-d_k^{r_k}),
$$
$a_k=(a_k^{ij})$ are matrices of homogeneous forms
with 
\begin{equation}\label{degar}
\deg a_k^{ij}=d_k^j -d_{k-1}^i
\end{equation}
\eqref{kongo} is exact, and 
the cokernel of the right-most mapping is precisely  $S/J_X$. 
Since $0$ is not an associated prime ideal of $J_X$ it follows from 
\cite[Corollary 20.14]{Eis} that one can
choose \eqref{kongo}  such that  $N_0\le N$.
Our integral  formulas are  explicitly  constructed out of 
a resolution \eqref{kongo}.

Recall that the (Castelnuovo-Mumford)
{\it regularity}  of $X$ is defined as the regularity of the ideal $J_X$ which 
turns out to be $1$ plus the regularity of the module $S/J_X$, so that
$$
\reg X=\max_{k,i} (d_k^i-k) +1,
$$
if \eqref{kongo} is a minimal free resolution of  $S/J_X$, cf.,  \cite[Ch.\ 4]{Eis2}.
It is well-known,  see, e.g., \cite[Proposition~4.16]{Eis2}, 
that    
\begin{equation}\label{eis1}
H^{q}(X,L^{s-q})=0,\  s\ge \reg X-1, \ q\ge 1, 
\end{equation}
and that the natural mapping
\begin{equation}\label{eis2}
\Ok(\P^N,\Ok(s))\to \Ok(X,L^s)
\end{equation}
is surjective for $s\ge\reg X-1$.

\smallskip
In \cite[Example 3.4]{AN} is described an extension operator that provides
an explicit proof of surjectivity of \eqref{eis2}, in case $X$ is reduced.
The non-reduced case is obtained in precisely the same way following the
ideas in \cite{AL}.  
By appropriate choices of  operators $\K$ we can give an explicit
proof of the vanishing of \eqref{eis1}, provided that $X$ is irreducible.
That is, we have

\begin{thm}\label{flicka} 
Let $X$ be a,  possibly non-reduced, irreducible,
subvariety of $\P^N$ of pure dimension $n$ and
assume that $s\ge\reg X-1$.
For each $q\ge 1$  there is an integral
operator $\K\colon \A^{q}(X,L^{s-q})\to\A^{q-1}(X,L^{s-q})$ such that
$\dbar\K\phi=\phi$ if  $\phi\in \A^{q}(X,L^{s-q})$ and  $\dbar\phi=0$.
 \end{thm}

In fact,  for fixed $q$ it is enough that  
$
s\ge\max_{\ell\le N-q} d_\ell^i- (N-q);
$
this follows from the proof below.
%
When $X$ is not irreducible a slightly less sharp version of the theorem still holds, see
Proposition \ref{patron}.  

\smallskip

Koppelman formulas
on $\Pk^n$ were found by G\"otmark, \cite{EG1}, and on more general symmetric spaces in \cite{GSS}.  
In \cite{He} explicit formulas for the
$\dbar$-equation are used on a smooth Riemann surface embedded in $\P^2$,
for $L^1$-estimates.  Similar formulas were also introduced in \cite{Helge},
cf.,  Section \ref{helgesson} below.
Koppelman formulas for global, even non-reduced, complete intersections are 
constructed in
the recent papers \cite{HePo3, HePo4}, cf.,  Section \ref{globcom} below.

As already mentioned, the main novelty in this paper is Koppelman formulas for
an arbitrary embedded projective variety $X$. We think that these formulas
will be of interest even when $X$ is smooth.
We prove Theorem \ref{flicka} as an illustration
of the utility and indicate some other applications in
Section \ref{negcurv}.   We hope that our Koppelman formulas will
be useful for other purposes as well.

\smallskip
In Section \ref{formelsec} we describe, based on \cite{EG1,AG,AN}, how one can obtain
weighted integral formulas on $\Pk^N$.  
We need some elements from residue theory that we have collected in
Section \ref{residue}.  
In Sections \ref{blip1} and \ref {blip2} we then describe the construction of our Koppelman
formulas on a pure-dimensional subvariety. In order to keep the technicalities 
on a reasonable level, we restrict here to the case where $X$ is reduced.
The reader who is mainly interested in the smooth case can just think that
$X_{sing}$ is empty, in that way avoiding a lot of technicalities.
In Section \ref{nonred} we discuss the non-reduced case.

\begin{remark}
The local study of the $\dbar$-equation on non-smooth spaces by $L^2$-methods was initiated
by Pardon and Stern, \cite{P}, \cite{PS1},\cite{PS2},  and has been developed  by a number of 
authors since then, notably Fornaess, Gavosto, Ovrelid, Ruppenthal, 
Vassiliadou, see., e.g.,  \cite{FG}, \cite{FOV1}, \cite{FOV2},  \cite{OR},
\cite{OV1}, \cite{OV2}, \cite{R3}, \cite{R5}.

The equation has also been studied by local and semi-global integral formulas
in, e.g.,   \cite{AS0,AS1}, \cite{HPo},  \cite{RuppZ}, \cite{LR}, \cite{LR2}.    
\end{remark}

\noindent {\bf Acknowledgement.}
The basic idea of this paper was used by P.\ Helgesson already in 2010;  he 
skillfully worked out the details in the special but nontrivial 
case when $X$ is a smooth Riemann surface   
in $\P^2$, \cite{Helge}.
We are grateful to Helgesson for valuable discussions on these matters. 
We also would like to thank the referee for careful reading and pointing out several
mistakes.

\section{Integral  representation on  $\Pk^N$}\label{formelsec}
We first describe how one can generate weighted Koppelman formulas on
$\Pk^N$ for sections of a  holomorphic vector bundle $F\to\Pk^N$.  
This is an adaption of an idea from \cite{A1} to $\Pk^N$, following
\cite{AG,AN};  see also \cite{EG1}.

\smallskip
Let $\pi\colon \C^{N+1}_z\setminus\{0\}\to \P^N_z$ be the natural projection and let $\U\subset\Pk^N$ be an open set. Recall that 
a form $\xi$ in $\pi^{-1}\U\subset \C^{N+1}_z\setminus\{0\}$ is projective, 
i.e., the pull-back of a form $\xi$ in $\U$,
if and only if $\xi$  is homogeneous and $\delta_z\xi=\delta_{\bar z}\xi=0$, where
$\delta_z$ and $\delta_{\bar z}$ are  interior multiplication by
$\sum_0^N z_j(\partial/\partial z_j)$ and its conjugate, respectively. 
We will identify forms in $\U$ by projective forms in $\pi^{-1}\U$.

\smallskip
Let  $\Ok_z(k)$ denote
the pullback of $\Ok(k) \to \Pn_z$ to $\Pn_\zeta\times \Pn_z$ under the 
projection $\Pn_\zeta\times \Pn_z \to \Pn_z$ and define $\Ok_\zeta(k)$
analogously.

\smallskip
\noindent {\it Throughout this paper we will only consider forms and currents that only contain
holomorphic differentials with respect to $\zeta$, whereas anti-holomorphic differentials
with respect to both $z$ and $\zeta$ may occur.  }
\smallskip

Notice that 
$$
\eta = 2 \pi i \sum_0^N z_i \frac{\partial}{\partial\zeta_i}, 
$$
is a section of $\Ok_z(1) \otimes \Ok_\zeta (-1)
\otimes T_{1,0}(\Pn_\zeta)$ on $\Pn_\zeta\times \Pn_z$.  
Contraction $\delta_\eta$ by $\eta$ defines a mapping
$$
\delta_\eta : \Cu^{\ell+1,q}(\Ok_\z(k) \otimes \Ok_z(j)) \to
  \Cu^{\ell,q}(\Ok_\z(k-1) \otimes \Ok_z(j+1)),
$$
where $\Cu^{\ell,q}(\Ok_\z(k) \otimes \Ok_z(j))$ denotes the
sheaf of currents of bidegree $(\ell,q)$ that take  values in $\Ok_\z(k) \otimes \Ok_z(j)$. 
Notice that $\delta_\eta$ only affects holomorphic differentials with respect to
$\zeta$.
Given a vector bundle $L\to\P^n_\zeta\times\P^n_z$, let
$$
\L^{\nu}(L)=\bigoplus_{j}\Cu^{j,j+\nu}(\Ok_\z(j) \otimes \Ok_z(-j)\otimes  L).
$$
If \[
\nabla_\eta = \delta_\eta - \dbar,
\]
then  $\nabla_\eta: \L^{\nu}(L) \to \L^{\nu+1}(L)$ and
$\nabla_\eta^2 = 0$.  
Furthermore, if $L'$ is a line bundle and $\phi, \psi$ are sections of $\L^\nu(L)$ and $\L^{\nu'}(L')$,
respectively, 
then $\phi\w\psi$ is a section of $\L^{\nu+\nu'}(L\otimes L')$, and
$$
\nabla_\eta(\phi\w\psi)=\nabla_\eta\phi\w\psi + (-1)^{\deg\phi}\phi \w \nabla_\eta\psi.
$$
Notice that 
$$
b=\frac{|\zeta|^2\bar z\cdot d\zeta-(\bar z\cdot\zeta)\bar\zeta\cdot d\zeta}
{|\zeta|^2|z|^2-|\bar\zeta\cdot z|^2}
$$
is a $(1,0)$-form on $\P^N_\zeta\times\P^N_z\setminus\Delta$
with values in $\Ok_\z(1) \otimes \Ok_z(-1)$
such that $\delta_\eta b=1$; here $\Delta$ is the diagonal in
$\P^N_\zeta\times\P^N_z$.

\begin{lma}  The form
$$
B=\frac{b}{\nabla_\eta b}=\frac{b}{1-\dbar b}=b+ b\w\dbar b+\cdots +b\w(\dbar b)^{N-1}
$$
is an  integrable section of $\L^{-1}$ and 
\begin{equation}\label{koppar}
\nabla_\eta B=1-[\Delta]_{d\zeta},
\end{equation}
where the last term is the component of the current of integration
$[\Delta]$ that has full degree $N$ in $d\zeta$.
\end{lma}

\begin{proof}
Let us consider the affinization where $\zeta_0\neq 0$.
In the affine coordinates $\zeta'_j=\zeta_j/\zeta_0$, $j=1, \ldots, N$, and the frame $z_0/\zeta_0$ for 
$\Ok_z(1)\otimes\Ok_\zeta(-1)$, we have
$$
\eta=2\pi i\sum_1^N(\zeta_j'-z_j')\frac{\partial}{\partial\zeta_j'}.
$$
Moreover, it is readily checked that
$|b|\le C |\zeta'-z'|$ and $|\delta_\eta b|\ge C |\zeta'-z'|^2$, and therefore
\eqref{koppar} follows, cf., \cite[Example~4]{A1}. 
\end{proof}

Given a vector bundle $F\to\P^N$, let $F_z$ denote the pullback of $F$  to $\Pn_\zeta\times \Pn_z$ under the
natural projection $\Pn_\zeta\times \Pn_z\to \Pn_z$  
and define  $F_\zeta$  analogously.
A  {\it weight} with respect to $F$  
is a smooth section  $g$ of $\L^{0}(\Hom(F_\zeta,F_z))$ such that $\nabla_\eta g =
  0$ and $g_{0} = I_F$ on the diagonal in $\P^N_\zeta\times\P^N_z$,
where $g_{0}$ denotes the term in $g$ with bidegree $(0,0)$. In general,
we let lower index on a form denote degree with respect to  holomorphic 
differentials of $\zeta$.  
Notice that if $g$ is a weight with respect to $F$, then from \eqref{koppar} we get
\begin{equation}\label{koppar2}
\nabla_\eta(g\w B)=(g-[\Delta]_{d\zeta})I_F.
\end{equation}
Identifying terms of full degree in $d\zeta$ thus
$$
\dbar(g\w B)_N=([\Delta]_{d\zeta}-g_N)I_F.
$$
By Stokes' theorem we get  the following  Koppelman formula,
cf., \cite{EG1} and \cite{GSS}.


\begin{prop} \label{hatsuyuki}
Let $g$ be a weight with respect to $F\to\P^N$.
Then for $\phi\in\E^{0,q}(\P^N,F\otimes\Ok(-N))$ 
we have
\begin{equation}\label{kopp0}
\phi(z)=\dbar_z \int_\zeta (g\w B)_N\w \phi+\int_\zeta(g\w B)_N\w \dbar\phi
+\int_\zeta g_N\w \phi.
\end{equation}
\end{prop}


\begin{ex} \label{kant}
It is easy to check that
\[
\alpha = \alpha_{0,0} + \alpha_{1,1} = \frac{z \cdot \bar \z}{|\z|^2} - \dbar
\frac{\bar \z \cdot d\z}{2 \pi i |\z|^2}
\]
is a projective form in $\L^0(\Ok_\zeta(-1)\oplus\Ok_z(1)))$
such that
\begin{equation}\label{alfasluten}
\nabla_\eta\alpha=0.
\end{equation}
Since $\alpha_{0}$ is equal to $I_{\Ok(1)}$ on the diagonal, $\alpha$ is a weight
with respect to $F=\Ok(1)$.  For each natural number $\rho$ therefore
$g=\alpha^\rho$
is a weight with respect to $F=\Ok(\rho)$. 
For $\ell\ge -N$ we  thus  have the Koppelman formula
\begin{equation}\label{koppleri}
\phi(z)=\dbar_z \int_\zeta (\alpha^{\ell+N}\w B)_N\w \phi+\int_\zeta(\alpha^{\ell+N}\w B)_N\w \dbar\phi
+\int_\zeta (\alpha^{N+\ell})_N\w \phi.
\end{equation}
Since $\alpha$ is holomorphic in $z$ and has no differentials $d\bar z$,
the last term in \eqref{koppleri} vanishes if  $q\ge 1$, 
and for degree reasons also if  $q=0$ and $\ell\le -1$.
\end{ex}

We thus have an explicit proof of the well-known vanishing 
$H^{0,q}(\P^N, \Ok(\ell))=0$
for $\ell\ge -N$ if $1\le q\le N$, and for $q=0$ if $\ell\le -1$.

\begin{remark}\label{uppdrag}
Using  the transposed integral operators 
in Example~\ref{kant},
we get an explicit proof of the vanishing 
$H^{N,q}(\P^N,\Ok(\ell))=0$ for $\ell\le N$, $0\le 1\le N-1$. 
Since $K_{\P^N}=\Ok(-N-1)$ we therefore get that    
$H^{0,q}(\P^N,\Ok(\ell))=0$ for $\ell\le -1$, $0\le q\le N-1$.
\end{remark}

Let us now consider a weight with respect to $\Ok(-1)$.
Let
$$
\tau=\frac{d\bar z\cdot d\zeta}{2\pi i|z|^2}=\frac{ \sum_0^N d\bar z_j\w d\zeta_j}{2\pi i|z|^2}.
$$
Then 
$$
\beta=2\pi i\Big[\delta_\zeta\frac{\delta_{\bar z}\tau}{1-\tau}\Big]=
2\pi i\Big[\frac{\delta_\zeta\delta_{\bar z}\tau}{1-\tau}+
\frac{\delta_\zeta\tau\w\delta_{\bar z}\tau}{(1-\tau)^2}\Big]
$$
is a projective form. More explicitly, 
\begin{equation}\label{palle}
\beta= \frac{\bar z\cdot\zeta}{|z|^2}+\frac{\bar z\cdot\zeta}{|z|^2}\frac{d\bar z\cdot d\zeta}
{2\pi i|z|^2} +\frac{\bar z\cdot d\zeta \w \zeta\cdot d\bar z}{2\pi i |z|^4}+ \cdots,
\end{equation}
so each term has the same degree in $d\zeta$ as in $d\bar z$, and
is holomorphic in $\zeta$.

\begin{prop} 
The form $\beta$ is a weight with respect to
$\Ok(-1)$.
\end{prop}

\begin{proof}
Clearly, cf.~\eqref{palle}, $\beta_{0,0}=I_{\Ok(-1)}$ on the diagonal.
We claim that 
\begin{equation}\label{hoppet}
\nabla_\eta \frac{\delta_{\bar z}\tau}{1-\tau}=1.
\end{equation}
Since $\delta_\zeta$ anti-commutes with $\nabla_\eta$ and $\delta_\zeta 1=0$
the proposition follows.
To see \eqref{hoppet},  notice that 
$$
\delta_\eta\tau=-\frac{z\cdot d\bar z}{|z|^2},
\quad
\dbar\tau=\delta_\eta\tau\w\tau, \quad
\delta_\eta\delta_{\bar z}\tau=1,
$$
so that
$$
\dbar\delta_{\bar z}\tau=\tau+\delta_\eta\tau\w\delta_{\bar z}\tau=
\tau+\gamma.
$$
Therefore,
\begin{equation}\label{ole}
\nabla_\eta \delta_{\bar z}\tau=1-\tau-\gamma,
\end{equation}
and
\begin{equation}\label{dole}
\nabla_\eta\tau=\delta_\eta\tau\w(1-\tau),
\end{equation}
so that
\begin{equation}\label{doff}
\nabla_\eta\tau\w\delta_{\bar z}\tau=(1-\tau)\w\gamma.
\end{equation}
It follows that
$$
\nabla_\eta\frac{\delta_{\bar z}\tau}{1-\tau}=
\frac{1-\tau-\gamma}{1-\tau}+\frac{\nabla_\eta\tau\w\delta_{\bar z}\tau}
{(1-\tau)^2}= 1,
$$
in view of \eqref{ole}, \eqref{dole}, and \eqref{doff}.
\end{proof}

Thus for $(0,q)$-forms  $\phi$ with  values in $\Ok(\ell)$, $\ell\le -N$,
we get from Proposition \ref{hatsuyuki}  the Koppelman formula
$$
\phi(z)=\dbar\int_\zeta(B\w\beta^{-N-\ell})_N\w\phi+
\int_\zeta(B\w\beta^{-N-\ell})_N\w\dbar\phi+\int_\zeta (\beta^{-N-\ell})_N\w\phi.
$$
For degree reasons the last term vanishes
if $0\le q\le N-1$, so we get back the well-known vanishing
$H^{0,q}(\P^N,\Ok(\ell))=0$ for $\ell\le -N$. 
In case $q=N$, the obstruction term vanishes when $\ell=-N$,
and when $\ell\le -N-1$   it vanishes if and only if 
\begin{equation}\label{grass}
\int \psi\w \phi=0
\end{equation}
for  each holomorphic $(N,0)$-form with values in $\Ok(-\ell)$. That is,
$\dbar v=\phi$ is solvable if and only if \eqref{grass} holds.
Of course this is
precisely what we get by considering the transposed operators with the weight $\alpha$, 
cf., Remark \ref{uppdrag}.  However, in the non-smooth case we have no obvious
canonical bundle so we cannot consider transposed operators in the same simple way;
therefore this weight $\beta$ will play a role.

\smallskip

For future reference we prove

\begin{prop}\label{stork}
The forms 
$$
\gamma_j=\delta_\zeta\Big[\frac{\delta_{\bar z}\tau}{1-\tau}\w d\zeta_j\Big]
$$
are projective and
\begin{equation}\label{gottis}
\nabla_\eta \gamma_j=\beta z_j-\zeta_j.
\end{equation}
\end{prop}

\begin{proof}
Clearly $\delta_\zeta \gamma_j=0=\delta_{\bar z}\gamma_j$  and thus 
$\gamma_j$ is a projective form.
By \eqref{hoppet} we have that 
$$
\nabla_\eta\Big[\frac{\delta_{\bar z}\tau}{1-\tau}\w d\zeta_j\Big]=
d\zeta_j-2\pi i\frac{\delta_{\bar z}\tau}{1-\tau}z_j.
$$
Since $\nabla_\eta$ and $\delta_\zeta$ anti-commute, 
the proposition follows.
\end{proof}

\section{Some preliminaries}\label{residue}

Let $X$ be any reduced analytic space of pure dimension $n$. By definition there is, locally, some embedding
$i\colon X\to \Omega\subset\C^N$.  Let $\J_X\subset \Ok_\Omega$ 
be the ideal sheaf of holomorphic functions in $\Omega$ that vanish on $X$. Then
the sheaf of holomorphic functions on $X$, the structure sheaf $\Ok_X$, is represented
as $\Ok_X=\Ok_\Omega/\J_X$.
If $\Phi$ is a smooth form in $\Omega$
we say that $\Phi$ is in $\Kers i^*$ if  $i^*\Phi$ vanishes on $X_{reg}$.
We define the sheaf 
$$
\E_X^{p,*}=\E_\Omega^{p,*}/\Kers i^*.
$$
of smooth forms on $X$, and have
a natural mapping $i^*\colon \E_\Omega^{p,*}\to \E_X^{p,*}.$
One can prove that $\E_X^{p,*}$ so defined is independent of the choice
of embedding and is thus an intrinsic sheaf on $X$.
We define the sheaf  $\Cu_X^{p,*}$ of currents as the dual of
$\E_X^{n-p,n-*}$. More concretely this means that currents $\tau$ in 
 $\Cu_X^{p,*}$ are
identified with  currents $i_*\tau$ in  $\Cu^{p+N-n,*+N-n}_\Omega$ such that
$i_*\tau$ vanish on $\Kers i^*$ so that $\tau.i^*\Phi=i_*\tau.\Phi$ for test forms
$\Phi$.  
Clearly $\dbar$ is defined on smooth forms
and extends to currents by duality. Also the wedge product $\phi\w\tau$ is well-defined
as long as at least one of the factors is smooth. Thus the currents form a
module over the smooth forms.

\smallskip

We say that a current in $\C^M_s$  of the form
$$
\frac{\gamma}{s_1^{a_1}\cdots s_r^{a_r}}\dbar\frac{1}{s_{r+1}^{a_{r+1}}}\w\cdots
\dbar\frac{1}{s_{r'}^{a_{r'}}},
$$
where $\gamma$ is a test form, is {\it elementary}. 
A current $\tau$ on $X$ is {\it pseudomeromorphic} if locally it is
a finite sum of direct images under holomorphic mappings of elementary currents;
see, e.g., \cite{AW3} for a precise definition and basic properties.  
The \pmm\  currents form a sheaf $\PM_X$ that is closed under multiplication by $\E_X^{p,*}$ and  the action of $\dbar$.  
Given a \pmm\ current $\tau$ in an open set $\U$ and a subvariety $V\subset \U$, the
natural restriction of $\tau$ to $\U\setminus V$ has a canonical  extension to
a \pmm\ current $\1_{\U\setminus V}\tau$  such that
$$
\1_V\tau:=\tau-\1_{\U\setminus V}\tau
$$
has support on $V$.  If $\xi$ is a smooth form, then
\begin{equation}\label{pm3}
\xi \w \1_V\tau=\1_V \xi\w\tau.
\end{equation}
Let $\chi$ be a smooth function on $[0,\infty)$ that is $0$ in a \nbh of $0$ and
$1$ in a \nbh of $\infty$ and let $h$ be a tuple of holomorphic functions,
or a section of some holomorphic Hermitian vector bundle such that
the zero set of $h$ is precisely $V$. Then
\begin{equation}\label{pm0}
\1_{\U\setminus V}\tau=\lim_{\delta\to 0} \chi(|h|/\delta)\tau.
\end{equation}

We say that a current $a$ in $X$ is {\it almost semi-meromorphic}, $a\in ASM(X)$, if
there is a smooth modification $\pi\colon X'\to X$, a generically
nonvanishing holomorphic section $\sigma$
of a line bundle $L\to X'$ and a smooth $L$-valued form $\gamma$ such that
$$
a=\pi_*(\gamma/\sigma).
$$
Let $ZSS(a)$ be the smallest analytic subset of $X$ such that
$a$ is smooth in $X\setminus ZSS(a)$. It follows that $ZSS(a)$ has positive codimension.
Clearly an almost semi-meromorphic $a$ is \pmm.

\begin{prop}[Theorem 4.8 in \cite{AW3}]\label{pm1}
Given any \pmm\ $\tau$ and $a\in ASM(X)$ the current
$a\w\tau$ a priori defined in $X\setminus ZSS(a)$ has a unique \pmm\ extension
to a \pmm\ current in $X$, also denoted $a\w\tau$, such that 
$\1_{ZSS(a)} a\w\tau =0$.  
\end{prop}

Pseudomeromorphic currents have some important geometric properties,
see, e.g., \cite{AW3}:

\begin{prop}\label{pm2}
Assume that the  \pmm\ current $\tau$ has support on a germ of an analytic variety $V$.

\smallskip
\noindent (i) If  the holomorphic function $h$ vanishes on $V$, then $\bar h \tau=0$ and $d\bar h\w \tau=0$.  

\noindent (ii)  If $\tau$ has bidegree $(*,p)$ and $V$ has codimension $\ge p+1$, then $\tau=0$. 
 \end{prop}
We refer to (ii) as the {\it dimension principle}.



\section{A structure form associated to $X$}\label{blip1}

Let $i\colon X\to \Pk^N$ be a reduced subvariety 
of pure dimension $n$, and let 
\eqref{kongo} be  a free graded resolution of the $S$-module $M_X=S/J_X$.
In particular, then $a_1=(a^{11}_1,\ldots, a^{1 r_1}_1)$ is a tuple of 
homogeneous forms that define the homogeneous ideal $J_X$
in the graded ring $S=\C[z_0,\ldots,z_N]$.  
Let $E_{k}^j$ be disjoint trivial line bundles over $\P^N$ 
with basis elements $e_{k,j}$   and let
$$
E_k=\big(E^1_k\otimes \Ok(-d^1_k)\big)\oplus\cdots 
\oplus \big(E^{r_k}_k\otimes \Ok(-d_k^{r_k})\big), \quad  E_0\simeq\C.
$$
Then 
\begin{equation}\label{ecomplex}
0\to E_{N_0}\stackrel{a_{N_0}}{\longrightarrow}\ldots\stackrel{a_3}{\longrightarrow} 
E_2\stackrel{a_2}{\longrightarrow}
E_1\stackrel{a_1}{\longrightarrow}E_0\to 0
\end{equation}
is a complex of vector bundles over $\P^N$
that is pointwise exact outside $Z$,  and  the corresponding complex
of locally free sheaves 
\begin{equation}\label{pulla}
0\to \Ok(E_{N_0})\stackrel{a_{N_0}}{\longrightarrow}\ldots\stackrel{a_3}{\longrightarrow} 
\Ok(E_2)\stackrel{a_2}{\longrightarrow}
\Ok(E_1)\stackrel{a_1}{\longrightarrow}\Ok(E_0)
\end{equation}
over $\P^N$ is a resolution of  the sheaf $\Ok(E_0)/\J_X$,
where $\J_X\subset\Ok_{\Pk^N}$ is the ideal sheaf associated with $X$.
See, e.g.,  \cite[Section~6]{AW1}.
We equip $E_k$ with the natural  Hermitian metric
$$
|\xi(z)|^2_{E_k}=\sum_{j=1}^{r_k}|\xi_j(z)|^2 |z|^{2d^j_k}
$$
if  $\xi=(\xi_1,\ldots,\xi_{r_k})$, so that \eqref{ecomplex} becomes
a Hermitian complex. 
In \cite{AW1} were  introduced  \pmm\  currents 
$$
U=U_1+\ldots +U_{N_0}, \quad R=R_1+\ldots+ R_{N_0}
$$
on $\Pk^N$ associated to \eqref{ecomplex} with the following properties:
The currents $U_k$ are almost semi-meromorphic $(0,k-1)$-currents, smooth outside $X$,
that take values in $\Hom(E_0, E_k)\simeq E_k$,  and
$R_k$ are  $(0,k)$-currents with support on $X$, taking values in 
$\Hom(E_0, E_k)\simeq E_k$. Moreover, we have the relations
\begin{equation}\label{alban}
a_1 U_1=I_{E_0}, \quad  a_{k+1}U_{k+1}-\dbar U_k= -R_k,\ k\ge 1,
\end{equation}
which can be compactly written as
\begin{equation}\label{compact}
\nabla_a U=I_{E_0}-R
\end{equation}
if 
$$
\nabla_a=a-\dbar=a_1+a_2+\cdots a_{N_0}-\dbar.
$$
If $\Phi$ is  a  section of $\Ok=\Ok(E_0)$, then 
the current $R\Phi$ vanishes  if and only if  $\Phi$ is 
in  $\J_X$, see \cite[Theorem~1.1]{AW1}.

Let $X_k$ be the analytic subset of $\Pk^N$ where $a_k$ does not have optimal rank.
Then 
\begin{equation}\label{pater}
\cdots X_{k+1}\subset X_k\dots X_{N-n+1}\subset X_{sing}
\subset X=X_{N-n}=\cdots =X_1.
\end{equation}
Since $\J_X$ has pure dimension
\begin{equation}\label{pust}
\codim X_k\ge k+1, \quad k\ge N-n+1,
\end{equation}
 see \cite[Corollary 20.14]{Eis}.

By the dimension principle
$R_k=0$ for $k<N-n$. Moreover, there are almost semi-meromorphic
$\Hom(E_k,E_{k+1})$-valued $(0,1)$-currents $\alpha_{k+1}$, smooth outside
$X_{k+1}$,  such that
$$
R_{k+1}=\alpha_{k+1}R_k
$$
there.  By \eqref{pust} and the dimension principle it follows that this equality must hold 
across $X_{k+1}$ if the right hand side is interpreted in the sense of Proposition \ref{pm1}.
By a simple induction argument, using \eqref{pust} and the dimension
principle,  it follows that 
\begin{equation}\label{pm4}
\1_{X_{sing}} R=0.
 \end{equation}

\begin{lma}\label{pucko}
If $\Phi$ is a smooth
$(0,*)$-form, then $i^*\Phi=0$ on $X_{reg}$ if and only if 
$\Phi R=0$.
\end{lma}

It follows that $R\phi$ is well-defined for $\phi$ in $\E_X^{0,*}$.   

\begin{proof}
Locally at a point $x\in X_{reg}$ we can choose coordinates $(z,w)$ such that
$X=\{w=0\}$.  By a Taylor expansion of $\Phi$ in $w$, using that $w_j R=\bar w_j R=
d\bar w_j\w R=0$, cf., Proposition \ref{pm2} (i), we find that
that $R\Phi=0$ if and only if $i^*\Phi=0$.  If $\Phi R=0$ on $X_{reg}$
it follows from \eqref{pm3} and \eqref{pm4} that $\Phi R=0$
identically.  
\end{proof}

Let 
$
\varOmega =\delta_z (dz_0\w\ldots\w dz_N)
$
be the unique, up to a multiplicative constant,
non-vanishing global $(N,0)$-form with values in $\Ok(N+1)$.
From \cite[Proposition 3.3]{AS1} we have
\begin{prop}  There is a unique almost
semi-meromorphic   current 
$\omega = \omega_0+\cdots +\omega_n$
on  $X$ that is smooth on $X_{reg}$,
 $\omega_\ell$ have bidegree $(n,\ell)$ and take values in $E^\ell:=i^*E_{N-n+\ell}$,
and 
\begin{equation}\label{deffo}
i_*\omega=\varOmega\w R, \quad  i_*\omega_\ell=\varOmega\w R_{N-n+\ell}.
\end{equation}
\end{prop}

We say that $\omega$ is a {\it structure form} on $X$.
%
%
%
For any smooth form $\xi$ on $\Pk^N$  there is a unique
form  $\sko(\xi)$ such that 
\begin{equation}\label{0skunk}
\sko(\xi)\w\varOmega=\xi_{N},
\end{equation}
where $\xi_{N}$ denotes  the components of $\xi$ of bidegree $(N,*)$.
From \eqref{deffo} and  \eqref{0skunk}  we have that
\begin{equation}\label{skunk}
\xi_{N}\w R=\sko(\xi)\w\varOmega\w R=i_*\big(\sko(\xi)\w\omega\big),
\end{equation}
where we in the last term, for simplicity, write $\sko(\xi)\w\omega$ rather then
$i^*\sko(\xi)\w\omega$.

\begin{lma}\label{snoklus}
Let $\chi(t)$ be a smooth function as in \eqref{pm0}. 
If $h$ is a holomorphic
section of a Hermitian vector bundle that
does not vanish identically  on any irreducible component of
$X$ and $\chi_\delta=\chi(|h|/\delta)$, then 
\begin{equation}\label{apfot}
\chi_\delta\omega\to \omega, \quad 
\dbar\chi_\delta\w \omega\to 0.
\end{equation}
\end{lma}

\begin{proof} 
Let $W$ be the zero set of $h$. 
Notice that $i_* \1_{X_{sing}} \omega=\1_{X_{sing}}i_*\omega=
\1_{X_{sing}} \varOmega\w R=\varOmega\w \1_{X_{sing}} R=0$ in view of
Lemma \ref{pucko}.  Thus $\1_{X_{sing}}\omega=0$, and hence $\1_W
\1_{X_{sing}}\omega=0$.
Since $\omega$ is smooth on $X_{reg}$ we have that 
$\1_W\1_{X_{reg}}\omega=0$. 
By simple computational rules, see, e.g., \cite{AW3}, we conclude that 
$\1_W \omega=0$.  In view of \eqref{pm0} 
thus the first part of \eqref{apfot} follows.  Notice that
\eqref{compact}  implies that $(a-\dbar)R=0$, 
and by  \eqref{deffo} thus $(a-\dbar)\omega=0$. 
Applying $(a-\dbar)$ to the first limit  in \eqref{apfot} now
the second one follows.
\end{proof}

\section{Koppelman formulas on a projective  variety}\label{blip2}

Let  $U^\lambda=|a_1|^{2\lambda} U$ and 
$R^\lambda=1-|a_1|^{2\lambda}+\dbar|a_1|^{2\lambda}\w U$.
Then $R^\lambda$ and $U^\lambda$ are as smooth as we may wish if
$\Re \lambda$ is sufficiently large. In particular,
$R^\lambda$ and $U^\lambda$ are well defined currents. Moreover, 
they admit analytic continuations to $\Re\lambda>-\epsilon$, and the
values at $\lambda=0$ are precisely $R$ and $U$, respectively, see \cite{AW1}.

Let $\rho$ be an integer.
Following  \cite[Definition 1]{AG} we say that  
$H= (H_k^\ell)$  is a Hefer morphism for  the complex $E_\bullet\otimes\Ok(\rho),a$, cf., \eqref{ecomplex}, if 
$H_k^\ell$ are smooth sections of
$$
\L^{-k+\ell}(\Hom(E_{\zeta,k}\otimes\Ok_\zeta(\rho),E_{z,\ell}\otimes\Ok_z(\rho))),
$$
$H_k^\ell=0$ for $k<\ell$, the term $(H_\ell^\ell)_{0}$ of bidegree $(0,0)$
is the identity $I_{E_\ell}$ on $\Delta$, and
\begin{equation}\label{hrel}
\nabla_\eta H_k^\ell =H_{k-1}^\ell a_k -a_{\ell+1}(z) H_k^{\ell+1}, 
\end{equation}
where $a_k$ stands for  $a_k(\zeta)$.
From \cite[Lemma 2.5]{AN} we have
\begin{lma} Assume that $H$ is a Hefer morphism for the complex 
$E_\bullet\otimes\Ok(\rho), a$.
For $\Re\lambda\gg 0$, the form 
\begin{equation}\label{gformel}
g^\lambda=a_1(z)H^1U^\lambda+H^0R^\lambda
\end{equation}
is a weight with respect to $\Ok(\rho)$. 
\end{lma}

Here $H^1U=\sum_j H^1_jU_j$ and $H^0R=\sum_j H^0_j R_j$.

\begin{prop}\label{basbox}
Let $F$ be holomorphic vector bundle over $\Pk^N$ and
let   $g$ be   a weight with respect to  $F$. 
Moreover, assume that $H$ is a Hefer morphism for $E_\bullet\otimes L^\rho$.
For $\phi\in\E^{0,q}(X, F\otimes L^{\rho-N})$ we have the Koppelman formula
\begin{multline}\label{kopp1}
\phi(z)=\dbar_z \int_\zeta (HR\w g\w B)_N\w \phi+\int_\zeta(HR\w g\w B)_N\w \dbar\phi
+ \\
\int_\zeta (HR\w g)_N\w \phi, \quad z\in X_{reg}.
\end{multline}
\end{prop}

By blowing up $\Pk^N\times \Pk^N$ 
along the diagonal one can verify that $B$ is almost semi-meromorphic. 
In view of Proposition \ref{pm1},
thus $(HR\w g\w B)_N\w \phi$ is a well-defined \pmm\  current on $\Pk^N\times \Pk^N$. Formally   \eqref{kopp1} means that
$$
\phi=\dbar \pi_*((HR\w g\w B)_N\w \phi)+\pi_*((HR\w g\w B)_N\w \dbar\phi)
+\pi_*((HR\w g)_N\w \phi),
$$
where $\pi$ is the projection $(\zeta,z)\mapsto z$.

\begin{proof}
Since $g^\lambda \w g$ is a weight with respect to $\Ok(\rho)\otimes F$, 
and $a_1(z)=0$ when $z\in X$, 
from Proposition \ref{hatsuyuki} we get \eqref{kopp1} with $R^\lambda$ instead of
$R$.   In view of (the proof of) \cite[Lemma 5.2]{AS1}, see also 
\cite[Lemma 9.5]{AL},  we can take $\lambda=0$ and so we get
the proposition, keeping in mind that 
the product $B\w \tau$ can be defined as the value
of $|\eta|^{2\lambda}B\w\tau$ at $\lambda=0$, in view of \cite[(2.2) and (2.3)]{AS1}.
\end{proof}

Let 
\begin{equation}\label{kloka}
\K\phi=\pi_*((HR\w g\w B)_N\w \phi), \quad \Pr\phi=\pi_*((HR\w g)_N\w \phi).
\end{equation}
Then we can write \eqref{kopp1} as 
\begin{equation}\label{kopp2}
 \phi=\dbar \K\phi +\K\dbar\phi +\Pr\phi
\end{equation}
for $z\in X_{reg}$.

In view of \eqref{skunk} we have that
\begin{equation}\label{kopp11}
\K \phi(z)=\int_X k(\zeta,z)\w\phi(\zeta), \quad
\Pr \phi(z)= \int_X p(\zeta,z)\w\phi(\zeta),
\end{equation}
where
\begin{equation}\label{valp}
k=\pm\sko (g\w B\w H)\w\omega,  \quad p=\pm\sko (g\w H)\w\omega.
\end{equation}
It is apparent from \eqref{kopp11} that $\K$ and $\Pr$ are intrinsic integral operators on
$X$. 
Locally they are precisely of the type in
\cite{AS1}, so it follows that $\K\phi$ is smooth on $X_{reg}$ if $\phi$ is smooth.
Moreover,  from \cite[Theorem 1.4]{AS1} we get: 
%
\begin{thm}\label{mojje}
Let $F$ be holomorphic vector bundle over $\Pk^N$ and
let   $g$ be   a weight with respect to  $F$ on  $X$.
Moreover, assume that $H$ is a Hefer morphism for $E_\bullet\otimes\Ok(\rho)$
and $\K$ and $\Pr$ are defined by \eqref{kopp11}. Then  
\begin{multline*}
\K\colon \A^{k+1}(X,E\otimes L^{\rho-N})\to \A^k(X,E\otimes L^{\rho-N}), \\
\Pr\colon \A^k(X,E\otimes L^{\rho-N}) \to \E^{0,k}(\Pk^N, E\otimes \Ok(\rho-N))
\end{multline*}
and the global
Koppelman formula \eqref{kopp2} holds on $X$ for $\phi\in\A^q(X,F\otimes L^{\rho-N})$.
\end{thm}


\section{The non-reduced case}\label{nonred}
Now assume that $i\colon X\to \Pk^N$ has pure dimension $n$ but is non-reduced. Then
we still have an ideal sheaf $\J_X\subset \Ok_{\Pk^N}$ that has pure dimension $n$
but $\J_X$ is no longer radical, i.e., there are nilpotent elements.  Still the structure
sheaf of $X$ has the representation $\Ok_X=\Ok_{\Pk^N}/\J_X$.
The underlying
reduced space $X_{red}$ is associated with  the radical ideal $\sqrt{\J_X}=\J_{X_{red}}$.  

Let $X_{reg}$ be the subset of $X$ where $X_{red}$ is smooth and in addition
$\J_X$ is Cohen-Macaulay. In a \nbh $\U$ of a point $x$ in $X_{reg}$, which is an open dense subset of $X$, we can choose local coordinates $(z,w)$ such that 
$X_{red}\cap\U=\{w=0\}$. It turns out, see, e.g., \cite{AL}, that there are monomials
$1, w^{\alpha_1}, \ldots, w^{\nu-1}$ such that each $\phi$ in $\Ok_X$ has a unique
representation
\begin{equation}\label{rep}
\phi=\hat\phi_0(z)\otimes 1+ \cdots + \hat\phi_{\nu-1}(z)\otimes w^{\nu-1}.
\end{equation}
Thus $\Ok_X$ has the structure of a free $\Ok_{X_{red}}$-module in $\U$.

We say that $\Phi$ in $\E^{0,*}_{\Pk^N}$ is in $\Kers i^*$ 
if in a \nbh of each point in $X_{reg}$, $\Phi$ is in the the subsheaf  of $\E_{\Pk^N}^{0,*}$ generated by
$\J_X, \bar \J_{X_{red}}$ and $d\bar \J_{X_{red}}$. 
%
As in the reduced case we define
$\E_X^{0,*}=\E_{\Pk^N}/\Kers i^*$, and again it is independent of the choice
of local embedding of $X$.
It turns out that at each point in $X_{reg}$ and coordinates $(z,w)$ as above we have
a unique representation \eqref{rep} of $\phi$ in $ \E_X^{0,*}$ where
$\hat\phi_j$ are in $\E_{X_{red}}^{0,*}$.  

We define the sheaf of $(n,*)$-currents on $X$ as the dual of 
 $\E_{X}^{0,n-*}$, 
so that such a current $\tau$ is represented by a 
$(N,N-n+*)$-current $i_*\tau$ in $\Pk^N$ that is annihilated by $\Ker i^*$.

Basically all facts in Section \ref{blip1} now hold verbatim, except for that 
one has to replace $X$ by $X_{red}$ in \eqref{pater} and 
slightly modify the proof of Lemma \ref{pucko}. 
The existence of the current $\omega$ on $X$ such that \eqref{deffo}
holds follows from Lemma \ref{pucko} but in the non-reduced
case we give no meaning to that $\omega$ is almost semi-meromorphic
and smooth on $X_{reg}$. The first part of \eqref{apfot} just means
that $\1_W R=0$ and this follows from \cite[Corollry 6.3]{AL}.
The second part of \eqref{apfot} follows from the first part precisely as before.

Following \cite{AL} we can also make the construction 
of Koppelman  formulas in Section \ref{blip2}  and define
sheaves $\A^*$ on $X$ so that Theorem \ref{mojje} holds.



\begin{remark}\label{kobolt}
Recall that a holomorphic differential operator $L$ is Noetherian with respect to the ideal 
$\J$ if $L\phi$ vanishes on $Z$ if $\phi\in \J$. 
As in the local case, cf., \cite[Remark 6.6]{AS1}, 
there is a tuple $\L$ of
global Noetherian operators on $\Pk^N$ with almost semi-meromorphic  coefficients so that
$$
\omega.\xi=\int_Z \L\xi,
$$
cf., \cite[Theorem 4.1 and Proposition 5.1]{Asznajdman}.
\end{remark}

\section{Global solutions}\label{globsol}

To begin with  we  consider a  Hefer morphism, introduced in \cite{AG}, for the
complex  $E_\bullet\otimes\Ok(\rho), a$ for large $\rho$.
Let $E'$ denote the complex of trivial bundles over
$\C^{N+1}$ that we get from $E$, and let $A$ denote the
corresponding mappings (which then formally are
just the original matrices $a$).
Let  $\delta_{w-z}$ denote interior multiplication by
$$
2\pi i\sum_0^N (w_j-z_j)\frac{\partial}{\partial w_j}.
$$
in $\C^{N+1}_w\times\C^{N+1}_z$.

\begin{prop} \label{kowalski}
There exist $(k-\ell,0)$-form-valued 
mappings 
$$
h_k^\ell=\sum_{ij}(h_k^\ell)_{ij} e_{\ell i} \otimes e_{k j}^\ast :
\C^{N+1}_w\times\C^{N+1}_z\to\Hom(E'_k , E'_\ell),
$$
such that $h_k^\ell = 0$ for $k<\ell$, $h_\ell^\ell = I_{E'_\ell}$,
\be \label{fraser}
\delta_{z-w} h_k^\ell =h_{k-1}^\ell A_k(w) - A_{\ell+1}(z) h_k^{\ell+1},
\ee
and the cooefficients in the form 
$(h_k^\ell)_{ij}$ are homogeneous polynomials of degree
$d_{k}^j-d_{\ell}^i-(k-\ell)$. 
\end{prop}

Notice that 
\begin{equation}\label{gammaform}
\gamma_j=d\zeta_j-\frac{\bar\zeta\cdot d\zeta}{|\zeta|^2}\zeta_j
\end{equation}
is  a projective form and that 
\begin{equation}\label{olle2}
\nabla_\eta \gamma_j=2\pi i(z_j-\alpha \zeta_j). 
\end{equation}
Given $h_k^\ell$ in Proposition~\ref{kowalski} we let
$\tau^* h_k^\ell$ be the projectice form we obtain by replacing
$w$ by $\alpha\zeta$ and $dw_j$ by $\gamma_j$.
We then have 
\begin{equation}\label{tau}
\nabla_\eta \tau^*h=\tau^*(\delta_{w-z}h),
\end{equation}
in light of (\ref{olle2}) and \eqref{alfasluten}. 
It is proved in \cite{AG} that if
$$
\kappa_0=\kappa_0(X)=\max d_k^i,
$$
then
\[
H^\ell_k=\sum_{ij}(\tau^\ast h_k^\ell)_{ij}\w \alpha^{\kappa_0-d_k^j}
e_{\ell,i}\otimes e^*_{k,j}, 
\]
is a Hefer morphism for  \eqref{ecomplex} with $\rho=\kappa_0$.
Clearly,  $H$  is holomorphic in $z$.

Recall that $g=\alpha^\nu$ is a holomorphic weight with respect
to $F=\Ok(\nu)$ for $\nu\ge 0$.
From Proposition~\ref{basbox}
we  thus obtain explicit solutions to the $\dbar$-equation
in $L^\ell $ for $\ell\ge\kappa_0(X)-N$.

\begin{prop}\label{patron}
Assume that $X$ is a possibly singular projective  subvariety of $\Pk^N$
of pure dimension $n$,  and $s\ge \kappa_0-N$.
If $\phi$ is a  $\dbar$-closed section in $\A^q(X, L^s)$, $q\ge 1$,  then
$$
\K\phi(z)=\int (H_{\kappa_0}R\w \alpha^{s-\kappa_0+N}\w B)_N\w\phi
$$
is a solution in $\A^{q-1}(X, L^s)$ to $\dbar u=\phi$. 
\end{prop}

Notice that 
$
\kappa_0-N\le \max (d_k^i-k)\le \reg X-1.
$
Thus the proposition gives a weaker form of Theorem \ref{flicka}.

\begin{remark}\label{patron2}
The associated operator $\Pr$ is precisely the operator in \cite[Example 3.4]{AN} that
realizes the surjectivity of \eqref{eis2}. That is, if $\phi\in \Ok(X, L^s)$, then
$$
\phi(z)=\int (HR\w \alpha^{s-\kappa_0+N})_N\w\phi
$$
is a global section of $\Ok(s)\to \Pk^N$  that coincides with $\phi$ on $X$.
\end{remark}



\begin{ex}\label{proto}
Assume that  $X$ is a complete intersection, i.e., $J$ is generated
by homogeneous forms $a_1^1,\ldots, a_1^p$, of degrees $d^1,\ldots, d^p$, where $p=N-n$. 
Then the Koszul complex generated by $a_1^j$ provides a minimal free resolution, and
it is then easy to see that $\kappa_0= d^1+\cdots+d^p$, cf., Section \ref{globcom} below.
Moreover, by the adjunction formula
$$
K_X=K_{\Pk^N}\otimes\Ok(d^1+\cdots +d^p)=L^{\kappa_0-N-1}.
$$
Here $K_X$ is the Grothendieck dualizing sheaf, which in this case is a line bundle
that is generated by $\omega=\omega_0$. When  $X$ is smooth $K_X$ is just the usual canonical bundle.  If we define $(n,q)$-forms as $(0,q)$-forms with values in $K_X$, then
Proposition \ref{patron} gives an explicit realization of the vanishing
\begin{equation}\label{van2}
H^{n,q}(X,L^\ell)=0,\quad 1\le q\le n, \quad \ell\ge 1.
\end{equation}
If $X$ is smooth this follows precisely  from Kodaira's theorem,
since $L$ is strictly positive on $X$. 
\end{ex}

For the proof of Theorem \ref{flicka} we must make a more careful analysis of the kernels.
Let us introduce the notation
\begin{equation}\label{lar1}
\kappa_q=\kappa_q(X)=\max_{\ell\le N-q} d^i_\ell.
\end{equation}
Notice that
 \begin{equation}\label{lar2}
\kappa_q-N\le =\max_{\ell\le N-q} \big(d^i_\ell -(N-\ell) -\ell\big)\le\reg X-1-q.
\end{equation}

\begin{proof}[Proof of Theorem \ref{flicka}]
Notice that the section  
$h(\zeta,z)=\zeta\cdot\bar z/|z|^2$ of $\Ok_\zeta(1)\otimes\Ok_z(-1)$ is 
non-vanishing on $\Delta$.
Let $\chi(t)$ be a cutoff function as before and let  
$$
\chi_\delta:=\chi(|h|^2/\delta)=\chi\big(|\bar\zeta\cdot z|^2/|z|^2|\zeta|^2\delta\big).
$$
Here $|h|$ denote the natural norm of  the section $h$, whereas in the last term
$|\ |$ denotes norm of points in $\C^{N+1}$. 
For small $\delta$, $\chi_\delta$ is identically $1$ in a \nbh of $\Delta$ and thus
\begin{equation}\label{lilian}
g^\delta:=\chi_\delta-\dbar\chi_\delta\w B
\end{equation}
is a smooth weight (with respect to the trivial line bundle).  
For fixed $z$,    $g^\delta$ vanishes in a \nbh of the hyperplane $h=0$, and therefore
$$
\alpha^{-r}\w g^\delta
$$
is a smooth weight with respect to $\Ok(-r)$  for any $r$, though not holomorphic
in $z$. Now fix $q\ge 1$ and let $t=s-q$.
In particular, for $t\ge \kappa_q-N$ and $\phi\in\A^q(X,L^{t})$, $q\ge 1$,  with $\dbar\phi=0$
we have the formula
\begin{multline}\label{lar3}
\phi(z)=
\dbar\int_\zeta (HR\w\alpha^{t-\kappa_0+N}\w g^\delta\w B)_N\w \phi+\\
\int_\zeta (HR\w\alpha^{t-\kappa_0+N}\w g^\delta\w\phi)_N\w\phi=
\dbar \K^\delta\phi+\Pr^\delta\phi.
\end{multline}
We claim that $\K^\delta\phi$ tends to a current $\K\phi$ in $\A^{q-1}(X,L^t)$ and that $\Pr^\delta\phi\to 0$.  
Taking this for granted, the theorem follows in view of \eqref{lar2}.

To settle the claim we first consider the expression for $\Pr^\delta\phi$ in \eqref{lar3}. 
Since $\phi$ has bidegree $(0,q)$ only
components $H_{N-\ell}R_{N-\ell}$ with $\ell\ge q$ can occur in the integral. 
Thus the total power of $\alpha$ is 
$$
\kappa_0 - d_{N-\ell}^i + t-\kappa_0 +N\ge  
\kappa_0- d_{N-\ell}^i + \kappa_q-N
-\kappa_0+N=\kappa_q-d^i_{N-\ell}\ge 0
$$
in view of \eqref{lar1}. 
Thus 
$$
\Pr^\delta\phi(z)=\sum_{\ell\ge q}\int_X \omega_{n-\ell} \w \phi_{0,q}\w \vartheta\big(\xi_\ell
\w (\chi_\delta-\dbar\chi_\delta\w B)\big),
$$
where $\xi_\ell$ are smooth and holomorphic in $z$.  Since $q\ge 1$ we need some
antiholomorphic differentials with respect to $z$ and they must come from 
$\dbar\chi_\delta\w B$; hence we can forget about $\chi_\delta$.
Since $\chi_\delta=1$ in a \nbh of the diagonal, we can consider
$B$ as smooth. Thus we have to verify that
\begin{equation}\label{aptass}
\omega \w \phi \w\dbar\chi_\delta \to 0, \quad \delta\to 0.
\end{equation}
Since $X$ is irreducible, $h=0$ has positive codimension on $X$, and if $\phi$ is smooth
thus \eqref{aptass} holds in view of Lemma \ref{snoklus}. 
If $\phi$ is in $\A^q$, then it is in $Dom_X$, cf.,\cite{AS1,AL}, and then \eqref{aptass} follows from (the proofs of)
\cite[Lemma 4.1]{AS1} and \cite[Lemma 8.4]{AL}.  In fact, \eqref{aptass} can be 
reformulated as $\1_{h=0} \dbar( \omega\w \phi)=0$. 
 We conclude that $\Pr^\delta\phi\to 0$.    
Notice now that $B\w B=0$ so that 
$$
\K^\delta\phi=\int_\zeta \chi_\delta(HR\w\alpha^{t-\kappa_0+N}\w B)_N\w \phi.
$$
It is proved in \cite{AS1,AL} that 
$(HR\w\alpha^{t-\kappa_0+N}\w B)_N\w \phi$ is in the space $\W^{X\times\Pk^N}$,
and this implies that 
$$
\chi_\delta (HR\w\alpha^{t-\kappa_0+N}\w B)_N\w \phi\to
(HR\w\alpha^{t-\kappa_0+N}\w B)_N\w \phi.
$$
It follows that  
$$
\K^\delta\phi\to  \K\phi=\int_\zeta (HR\w\alpha^{t-\kappa_0+N}\w B)_N\w \phi.
$$
\end{proof}

 \subsection{Examples with negative curvature}\label{negcurv}

We now turn our attention to the case of negative curvature.
We define a Hefer morphism from $h^\ell_k$ in Proposition~\ref{kowalski}
by replacing
$w$ by $\zeta$,  $z$ by $\beta z$, and $dw_j$ by the $\gamma_j$ from
Proposition~\ref{stork}.  The morphism so obtained is a Hefer 
morphism for \eqref{ecomplex}
(i.e., for  $E_\bullet\otimes\Ok(\rho), a$ with $\rho=0$). 
This is verified
in the same way as \cite[Proposition 4.4]{AG}.

This time $H$ is not holomorphic in $z$ but in $\zeta$ instead.  Let $\delta$ be the
depth of the ring $S/J$. This is a number, $0\le \delta\le n$, and choosing
\eqref{kongo} minimal, \eqref{ecomplex} will end up at $k=N-\delta$, which means
that $R=R_{N-n}+\cdots +R_{N-\delta}$.  The variety  $X$ is Cohen-Macaulay
precisely  when $\delta=n$.

\smallskip
From the Koppelman formula we get solutions to $\dbar$
(representation of the  cohomology in the smooth case) 
for $(0,q)$-forms $\phi$ with values in
$\Ok(\ell)$ for $\ell\le -N$ and 
thus solutions as soon as the obstruction term
\begin{equation}\label{ob}
\int (HR\w \beta^{-N-\ell})_N\w\phi
\end{equation}
vanishes. 
Notice that $HR$ has degree at most $N-\delta$ in $d\bar\zeta$
since  $\beta$ and $\gamma_j$ only contain holomorphic differentials
with respect to $\zeta$. Therefore \eqref{ob} must vanish
if $N-\delta+q<N$, i.e., $0\le q\le \delta-1$.

\begin{thm}\label{van2}
Assume that $X$ is a subvariety of $\P^N$ of pure dimension $n$
and $\ell\le -N$. Then for any $\dbar$-closed $(0,q)$-form 
$\phi\in\A_q(X, L^\ell)$,\  $0<q\le \delta-1$,
$$
\psi(z)=\int(HR\w \beta^{-N-\ell}\w B)_N\w \phi 
$$
is a solution in $\A^{q-1}(X, L^\ell)$ to $\dbar \psi=\phi$. 
\end{thm}

We thus have an explicit proof of the vanishing 
$H^{0,q}(X,\Ok(\ell))=0$
for $0\le q\le \delta-1$,  $\ell\le -N$.

\section{Global complete intersections}\label{globcom}

Let us compute the resulting formulas in case $i\colon X\to \Pk^N$ is a global complete
intersection as in Example \ref{proto}.
Let $f_1,\ldots, f_p$, $f_j=a_1^j$, be our given  homogeneous forms of  degrees $d^j$
from Example \ref{proto}, and recall that $p=N-n$. 
Assume that $E_0$ is the trivial line bundle  and let 
$$
E=E^1\otimes\Ok(-d^1)\oplus
\cdots\oplus E^{p}\otimes\Ok(-d^p),
$$
where $E^j$ are trivial line bundles.
Let $e_j$   be basis elements  for $E^j$ and let $e_j^*$ be the dual basis elements.
We  take  
$$
E_k=\Lambda^k E=\sum'_{|I|=k}\Ok(-(d^{I_1}+\cdots + d^{I_k}))E^{I_1}\otimes\cdots
\otimes E^{I_k}
$$
and $a_k\colon E_k\to E_{k-1}$ as interior multiplication by $f=\sum f_j e_j^*$. 
Now
$$
\sigma=\sum_j\frac{\overline{f_j(z)}}{|z|^{2d^j}}e_j/\|f\|^2
$$
is the section of $E$ with minimal norm such that $f\cdot \sigma=1$ outside $Z$, 
if $\|f\|=|f|_{E^*}$. Moreover,
$$
U=\|f\|^{2\lambda}\sum_{k=1}^{m} \sigma\w (\dbar \sigma)^{k-1} \Big|_{\lambda=0}
$$
and
\begin{equation}\label{Rdef}
R=1-\|f\|^{2\lambda}+
\dbar \|f\|^{2\lambda}\w \sum_{k=1}^m \sigma\w (\dbar \sigma)^{k-1}\Big|_{\lambda=0},
\end{equation}
cf., Section \ref{blip2} and \cite{A2};  here $|_{\lambda=0}$ means evaluation at $\lambda=0$ after
analytic continuation. 
Since $\codim Z=p$ the resulting residue current $R$
just consists of the term $R_p$; it 
coincides with the classical Coleff-Herrera product
$$
\dbar\frac{1}{f_p}\w\ldots \w \dbar\frac{1}{f_1}\w e_1\w\ldots\w e_p.
$$
We now compute Hefer morphisms   for the Koszul complex.     
Let $ \tilde h_j(w,z)$ be $(1,0)$-forms in $\C^{n+1}\times\C^{n+1}$ 
of polynomial degrees $d^j-1$ such that
$$
\delta_{w-z} \tilde h_j=f_j(w)-f_j(z)
$$
and let 
$
h_j=\tau^*\tilde h_j 
$
We only have to care about $k\le p$ so $\kappa_0 = d^1+\ldots +d^p$.
Then 
$$
H^\ell_k=\sum'_{|I|=\ell}\sum'_{|J|=k-\ell}\pm
h_{J_1}\w\ldots\w h_{J_{k-\ell}}\w e_I\otimes e^*_{IJ}\w\alpha^{\kappa_0-(d^{J_1}+\cdots +d^{J_{k-\ell}}
+d^{I_1}+\cdots +d^{I_k})}
$$
is a Hefer morphism.  The components of most interest for us are 
$H_k^0$ and $H_k^1$. Since 
$$
H_k^0=\sum'_{|J|=k}\pm
h_{J_1}\w\ldots\w h_{J_{k}}\w  e^*_{J}\w\alpha^{\kappa_0-(d^{J_1}+\cdots +d^{J_{k}})}
$$
it can be more compactly written  formally as 
$$
H_k^0=\alpha^{\kappa_0} \w (\delta_h)_k
$$
where $\delta_h$ denotes formal interior multiplication with
$$
h=\sum \alpha^{-d^j}\w h_j\w e_j^*
$$
and $(\delta_h)_k=(\delta_h)^k/k!.$
In the same way
$$
H_k^1=\alpha^{\kappa_0}\w N(\delta_h)_{k-1},
$$
where
$$
N=\sum_j \alpha^{-d^j}e_j\otimes e_j^*.
$$

\smallskip
Our description of  $U$, $H^1_k$ etc is just to illustrate what these currents look like
in the complete intersection case since they play a rule in the proofs above. As we have seen, however, in the final Koppelman formula only the term 
$$
H_p^0 R_p= h_1\w\ldots\w h_p\w 
\dbar\frac{1}{f_p}\w\ldots \w \dbar\frac{1}{f_1}
$$
of $HR$ occurs.  
It follows that the operator $\K$ in Proposition \ref{patron}, with 
\begin{equation}\label{stamband}
\kappa=s+N-(d^1+\cdots+ d^p)
\end{equation}
has the more explicit form  
\begin{equation}\label{puma1}
\K\phi(z)=\int  \big(\alpha^{\kappa}\w B\w h_1\w\ldots\w h_p\big)_N
\w \dbar\frac{1}{f_p}\w\ldots \w \dbar\frac{1}{f_1}\w \phi
\end{equation}
and the operator $\Pr$ in Remark \ref{patron2} is 
\begin{equation}\label{puma2}
\Pr\phi(z)=\int  \big(\alpha^{\kappa}\w h_1\w\ldots\w h_p\big)_N
\w \dbar\frac{1}{f_p}\w\ldots \w \dbar\frac{1}{f_1}\w \phi.
\end{equation}

\begin{prop} Assume that the projective space $i\colon X\to \Pk^N$ of codimension $p$
is defined by the homogeneous forms
$f_j$ on $\C^{N+1}$ of degree $d^j$, $j=1,\ldots,p$ and assume that
$s\ge  d^1+\cdots + d^p-N$. For $\phi\in \E^{0,k}(X,L^{s})$,
or $\phi\in \A^{k}(X,L^{s})$, we have
the Koppelman formula \eqref{kopp2}  with 
$\K$ and $\Pr$ defined by \eqref{puma1} and \eqref{puma2} and $\Pr$ vanishes
if $k\ge 1$.  
\end{prop}

\begin{remark}  
In \cite{HePo3}  similar Koppelman formulas are obtained on a,
not necessarily reduced, global complete intersection $X$ for $(0,*)$-forms
with values in $L^{d^1+\cdots + d^p -N-1}$ in the notation from Example \ref{proto}.
They construct Koppelman formulas on homogeneous subvarieties of
 $\C^{N+1}$, keep track of homogeneities and so obtain Koppelman formulas on $X$.
They use the same definition of $\dbar$ as we do. 
However, they only consider solutions to $\dbar u=\phi$ on $X_{reg}$ when $\phi$ is
smooth on $X$ and satisfies a condition $(*)$ that in general is stronger than $\dbar\phi=0$.
There is no discussion whether their solution has some meaning as a current across $X_{sing}$.
The condition $(*)$ on $\phi$ means that (locally) there is a smooth extension $\Phi$ to 
ambient space such that $\dbar\Phi$ is in $\E \J$.  Clearly this implies that 
$\dbar\phi=0$ on $X$ but in general the converse does not hold. In fact, consider 
a reduced hypersurface $X=\{a=0\}\subset \C^{n+1}$ so that $\J=(a)$. Then $(*)$ 
means that there is an extension $\Phi$ such that $\dbar\Phi=\xi a$ for some smooth form $\xi$. Then $0=a\dbar\xi$ and thus $\dbar\xi=0$; hence $\xi=\dbar\eta$ for some smooth $\eta$.  Now 
$\Phi-a\eta$ is $\dbar$-closed and therefore there is a
smooth solution to $\dbar \Psi = \Phi-a\eta$.  It follows that $\psi=i^*\Psi$
is a smooth solution to $\dbar\psi=\phi$.   
However,  it is well-known that there are
smooth $\phi$ with  $\dbar\phi=0$ such that $\dbar\psi=\phi$ has no smooth
solution, see, e.g., \cite[Example 1.1]{AS1}.  
\end{remark}

\subsection{The reduced case} 
Let us consider a more intrinsic-looking representation of $\K$ and $\Pr$ as in
\eqref{kopp11} and \eqref{valp}.  
In order to avoid a Noetherian operator,  cf., Remark \ref{kobolt}, let us in addition assume that $X$ is reduced. 
Let $A_1,\ldots,A_p$ be holomorphic vector fields on $\C^{N+1}$ such that
$$
\delta_{A_p}\cdots\delta_{A_1} df_1\w \ldots\w df_p=(2\pi i)^p,
$$
or equivalently, 
\begin{equation}\label{greta}
df_1\w\cdots df_p\w \delta_{A_p}\cdots\delta_{A_1}(d\zeta_0\w\ldots\w d\zeta_N)=(2\pi i)^p
d\zeta_0\w\ldots\w d\zeta_N.
\end{equation}
Notice that since $\delta_\zeta$ anti-commutes with $\delta_{A_j}$, 
\begin{equation}\label{greta2}
\omega':=\delta_{A_p}\cdots\delta_{A_1} \varOmega
\end{equation}
is a projective form.   Following the proof of \cite[Proposition 6.3]{AN}  we see that
$\omega'$ is a representative for the structure form on $X$, that is, 
\begin{equation}\label{greta3}
\omega=\omega_0=i^*\omega'.
\end{equation}

\begin{ex}
Let 
$$
\delta_A=\delta_{A_1}\cdots\delta_{A_p}, \quad \dbar\frac{1}{f}=\dbar\frac{1}{f_p}\w
\cdots\dbar\frac{1}{f_1}.
$$
Then
\begin{equation}\label{partout}
\delta_A\xi_N=\delta_A(\vartheta(\xi)\w\varOmega)=\pm \vartheta(\xi)\w\omega'.
\end{equation}
In view of \eqref{valp} we thus have that 
$$
\K\phi=\pm\int_X \delta_A(\alpha^\kappa\w B\w h)_N\w \phi, \quad
\Pr\phi=\pm\int_X \delta_A(\alpha^\kappa\w  h)_N\w \phi
$$
for $\phi\in\A^q(X,L^s)$, \ $s=\kappa+d^1+\cdots + d^p-N$, $\kappa\ge 0$.
\end{ex}

\subsection{Explicit formulas for a curve in $\P^N$}\label{helgesson}
Following \cite{Helge} we will now describe how one can find an especially simple expression for the kernel $k$ 
when $X$ is a curve.    Applying $\delta_\eta$ to 
$$
(\alpha^\kappa\w B\w h)_N=\vartheta(\alpha^\kappa\w B\w h)\w\varOmega,
$$
cf., \eqref{0skunk}, we get
\begin{equation}\label{snuva}
\delta_\eta(\alpha^\kappa\w B\w h)_{N}=\pm
\vartheta(\alpha^\kappa\w B\w h)\w\delta_\eta\varOmega.
\end{equation}
We claim that
\begin{equation}\label{koppargruva}
\delta_\eta(\alpha^\kappa\w B\w h)_{N}=(\alpha^\kappa\w h)_{N-1}
\end{equation}
on $X\times X$. 
In fact, consider the weight $g=\alpha^\kappa\w g^\lambda$, cf.,  \eqref{gformel}.
From \eqref{koppar2}, keeping $z$ on $X$ and 
taking $\lambda=0$ we get 
$$
\nabla_\eta(\alpha^\kappa\w B\w h)\w R=\alpha^\kappa\w h \w R
$$
outside the diagonal. Identifying terms of degree $N-1$ in $d\zeta$ we get 
$$
\delta_\eta(\alpha^\kappa\w B\w h)_{N}\w R=(\alpha^\kappa\w h)_{N-1} \w R.
$$
By the generalized Poincar\'e-Lelong formula, $R\w df=(2\pi i)^p [X]$,
we can conclude
that \eqref{koppargruva} holds on $X\times X$.  Combining \eqref{snuva} and \eqref{koppargruva} we
get that
\begin{equation}\label{huva}
\pm\vartheta(\alpha^\kappa\w B\w h)\w\delta_\eta\varOmega=
(\alpha^\kappa\w h)_{N-1}
\end{equation}
on $X\times X$. 

\smallskip
Let us now compute the kernel $k$ where $X$ is parametrized by
$[\zeta_0,\zeta_1]$.  Let $\delta_{\zeta_j}$
denote interior multplication by $\partial/\partial \zeta_j$. 
Notice that
$$
\delta_\eta\delta_{\zeta_2}\cdots \delta_{\zeta_N}\varOmega=2\pi i(\zeta_1z_0-\zeta_0z_1).
$$
If we apply 
$\delta_{\zeta_2}\cdots \delta_{\zeta_N}$ to \eqref{huva} we get 
\begin{equation}\label{huva2}
\pm\vartheta(\alpha^\kappa\w B\w h)=\frac{1}{2\pi i}\frac{\delta_{\zeta_2}\cdots \delta_{\zeta_N}(\alpha^\kappa\w h)_{N-1}}{\zeta_1z_0-\zeta_0z_1}.
\end{equation}
Notice that the denominator on the right hand side is smooth. 
Recalling that
$k(\zeta,z)=\pm\vartheta(\alpha^\kappa\w B\w h)\w\omega$
on $X\times X$, cf., \eqref{valp},
we get from \eqref{huva2} and \eqref{greta2}, cf., \eqref{greta3}, 
\begin{prop}
With the notation above we have 
the explicit formula
\begin{equation}\label{gamas}
k(\zeta,z)=\frac{1}{2\pi i}\frac{\delta_{\zeta_2}\cdots \delta_{\zeta_N}(\alpha^\kappa\w h)_{N-1}}{\zeta_1z_0-\zeta_0z_1} \delta_A\varOmega
\end{equation}
where $X$ is parametrized by $[\zeta_0,\zeta_1]$.  
\end{prop}

\subsection{Curves in $\Pk^2$}
Assume now that $N=2$ and $X=\{f=0\}$, where $f$ is a $d$-homogeneous form.  Thus $s=\kappa+d-2$.
Let $\tilde h(w,z)=h_0dw_0+h_1 dw_1+ h_2dw_2$ be a Hefer form for $f$; i.e., $h_\ell$ are $d-1$-homogeneous and
\begin{equation}\label{plats}
2\pi i \sum_{\ell=0}^2h_\ell(w,z)(w_\ell-z_\ell)=f(w)-f(z).
\end{equation}
Notice that for degree reasons, 
$$
(\alpha^\kappa\w h)_1=\alpha_0^\kappa\big( h_0(\alpha_0\zeta,z)\gamma_0+
h_1(\alpha_0\zeta,z)\gamma_1+h_2(\alpha_0\zeta,z)\gamma_2\big).
$$
In what follows, for simplicity, let us right $\alpha$ rather than $\alpha_0$.
If $\partial f/\partial \zeta_2$ is generically nonvanishing on $X$ and
$$
A=\frac{2\pi i}{\partial 
f/\partial \zeta_2}\frac{\partial}{\partial\zeta_2},
$$
then $\delta_A df=2\pi i$ (generically) on $X$. 
We get
\begin{equation}\label{gamas2}
k(\zeta,z)=\frac{1}{2\pi i}\frac{\zeta_0d\zeta_1-\zeta_1 d\zeta_0 } {\zeta_1z_0-\zeta_0z_1}
 \frac{2\pi i\delta_{\zeta_2} (\alpha^\kappa\w h)_1}{\partial 
f/\partial \zeta_2}
\end{equation}
Notice furthermore that  
\begin{equation}\label{rakost}
\delta_{\zeta_2} (\alpha^\kappa\w h)_1= \alpha^\kappa\Big(h_2(\alpha\zeta,z)-
\frac{\bar\zeta_2}{|\zeta|^2}\sum_{j=0}^2 h_j(\alpha\zeta,z) \zeta_j\Big).
\end{equation}

\begin{prop} With the notation above we have
\begin{equation}\label{pontus}
k(\zeta,z)=\frac{1}{2\pi i}\frac{\zeta_0d\zeta_1-\zeta_1 d\zeta_0 } {\zeta_1z_0-\zeta_0z_1}
 \frac{2\pi i \alpha^\kappa}{\partial 
f/\partial \zeta_2}\Big(h_2(\alpha\zeta,z)-
\frac{\bar\zeta_2}{|\zeta|^2}\sum_{j=0}^2 h_j(\alpha\zeta,z) \zeta_j\Big).
\end{equation}
\end{prop}

The second term in the bracket actually cancels out the singularity
if $X$ is smooth.

\begin{prop}\label{koppleri2}
If $X$ is smooth and 
$[\zeta_0,\zeta_1]$ are local homogeneous coordinates, then
\begin{equation}\label{arm}
k(\zeta,z)=\frac{1}{2\pi i}\frac{\zeta_0d\zeta_1-\zeta_1 d\zeta_0 } {\zeta_1z_0-\zeta_0z_1}\frac{\alpha^\kappa 2\pi i h_2(\alpha\zeta,z)}{\partial 
f/\partial \zeta_2}+ \cdots,
\end{equation}
where $\cdots$ is smooth and holomorphic in $z$.  
\end{prop}

Notice that $2\pi i h_2(\alpha\zeta,z)=\partial 
f/\partial \zeta_2$ on the diagonal.

\begin{proof} 
Differentiating \eqref{plats} with respect to $w_j$ gives
$$
2\pi i h_j(w,z)=\frac{\partial f}{\partial w_j}(w)-2\pi i\sum_\ell
\frac{\partial h_\ell}{\partial w_j}(w_\ell-z_\ell).
$$
Since $f$ is $d$-homogeneous,  
$$
\sum_{j=0}^2 \zeta_j\frac{\partial f}{\partial\zeta_j}=d \cdot f(\zeta).
$$
We conclude that
$$
\sum_0^2 \zeta_j\tilde h_j(\alpha\zeta_j,z)=d \alpha^{d-1} \cdot f(\zeta)+
\sum_0^2b_\ell(\zeta,\alpha\zeta,z)(\alpha\zeta_\ell-z_\ell),
$$
where $b_\ell(\zeta,w,z)$ are holomorphic, $d-2$-homogeneous 
in $(w,z)$ and 1-homogeneous in $\zeta$.   
Since $\zeta$ is on $X$, $f(\zeta)=0$.  We thus get \eqref{arm}
where 
$$
\cdots = \frac{1}{2\pi i}\frac{\zeta_0d\zeta_1-\zeta_1 d\zeta_0 } {\zeta_1z_0-\zeta_0z_1} B
$$
and
$$
B=\frac{\bar\zeta_2}{|\zeta|^2} \alpha^{\kappa} 
\sum_0^2b_j(\zeta,\alpha\zeta,z)(\alpha\zeta_j-z_j).
$$
Without loss of generality we may assume that  
$\zeta_0=z_0=1$ and that $\zeta_1$ is a local coordinate on $X$ so that
$\zeta'=g(\zeta_1)$. Since $\alpha=1$ on the diagonal we then have
that $B=B' (\zeta_1-z_1)$
where $B'$ is holomorphic in $z$.  After homogenization we get that
$B=B''(z_0\zeta_1-z_1\zeta_0),$ where $B''$ is holomorphic in $z$.
Thus the proposition follows.
\end{proof}

\begin{cor}
If $\zeta_1$ is a local coordinate, then taking $\zeta_1=\tau$, $z_1=t$,
and $\zeta_0=z_0=1$, we get
\begin{equation}\label{arm2}
k(\tau, t)=\frac{1}{ 2\pi i}\frac{d\tau } 
{\tau-t}\frac{\alpha^\kappa2\pi i h_2(1,\tau,\zeta_2; 1,t,z_2))}{\partial 
f/\partial \zeta_2}+ \cdots,
\end{equation}
where $\cdots$ is smooth and holomorphic in $t$.
\end{cor}

\begin{ex}
If $f(z)=z_0^3+z_1^3+z_2^3$, then $X$ is a smooth 
surface of genus $1$.   We can take
$$
\tilde h(z,w)=\sum_{j=0}^2(z_j^2+z_jw_j+w_j^2)dw_j/2\pi i.
$$
Then, where $[z_0,z_1]$ are homogeneous coordinates, 
\begin{multline*}
k(\zeta,z)=\alpha^\kappa \frac{1}{2\pi i}\frac{\zeta_0d\zeta_1-\zeta_1 d\zeta_0 } {\zeta_1z_0-\zeta_0z_1}\frac{1}{3\zeta_0^2}
\Big[z_2^2+\alpha z_2\zeta_2+\alpha^2\zeta_2^2 - \\
\frac{\bar\zeta_2}{|\zeta|^2}\Big(z_0^2\zeta_0+z_1^2\zeta_1+z_2^2\zeta_2
+\alpha(z_0\zeta_0^2+z_1\zeta_1^2+z_2\zeta_2^2)\Big)\Big].
\end{multline*}
\end{ex}

Even if $X$ is not smooth, by the same argument,
one can identify the principal term of the kernel
$k$.

\begin{ex}
The curve $X=\{z_1^3-z_2^2 z_0=0\}$ has a cusp singularity at $[1,0,0]$ and is smooth elsewhere. It is globally parametrized by
$$
\Pk^1\to X\subset \Pk^2, \quad  [t_0,t_1]\mapsto[t_0^3,t_1^2 t_0,t_1^3].
$$
We can choose the Hefer form
$$
2\pi i \tilde h(w,z)=z_2^2 dw_0 -(z_1^2+z_1 w_1 + w_1^2)dw_1 +(z_2+w_2)w_0 dw_2.
$$
By formula \eqref{pontus} we can now express $k$ completely in terms
of the parameters $[t_0,t_1]$ and $[\tau_0,\tau_1]$. However, we restrict to
considering the principal term.  Since we are primarily interested in the
singularity, we consider the standard affinization where $z_0=\zeta_0=1$.
By the recipe above, then  
\begin{equation}\label{pass}
k=\frac{1}{2\pi i}\frac{d\zeta_1}{\zeta_1-z_1}\alpha^\kappa\frac{z_2+\alpha\zeta_2}
{2\zeta_2}+\cdots
\end{equation}
In this case $\cdots$ is not smooth but at least the singularity is smaller
than in the leading term. Since $\alpha-1=\Ok(\zeta_1-z_1)+\Ok(\zeta_2-z_2)$, we can delete $\alpha$ as well in the leading term in \eqref{pass}.
We then have
$$
k=\frac{1}{2\pi i}\frac{(\zeta_2^2-z_2^2)d\zeta_1}
{(\zeta_1-z_1)(\zeta_2-z_2)2\zeta_2}+\cdots
$$
Letting 
$t=t_1/t_0$ and $\tau=\tau_1/\tau_0$ we are then left with
$$
k=\frac{1}{2\pi i} \frac{(\tau^6 - t^6) d\tau}
{(\tau^2-t^2)(\tau^3-t^3)\tau^2}+\cdots
$$
where the leading term is precisely the kernel in the last example in \cite[Section 8]{AS1}.
\end{ex}

\def\listing#1#2#3{{\sc #1}:\ {\it #2},\ #3.}

\end{document}